\documentclass[sn-mathphys-num]{sn-jnl}% Math and Physical Sciences Numbered Reference Style 
\usepackage{graphicx}%
\usepackage{multirow}%
\usepackage{amsmath,amssymb,amsfonts}%
\usepackage{amsthm}%
\usepackage{mathrsfs}%
\usepackage[title]{appendix}%
\usepackage{color}%
\usepackage{textcomp}%
\usepackage{manyfoot}%
\usepackage{booktabs}%
\usepackage{algorithm}%
\usepackage{algorithmicx}%
\usepackage{algpseudocode}%
\usepackage{listings}%
\theoremstyle{thmstyleone}%
\newtheorem{theorem}{Theorem}%  meant for continuous numbers
\newtheorem{lemma}[theorem]{Lemma}
\newtheorem{corollary}[theorem]{Corollary}

\theoremstyle{thmstyletwo}%

\theoremstyle{thmstylethree}%
\numberwithin{equation}{section}
\begin{document}

\title[Article Title]{Bilateral Bailey pairs and Rogers--Ramanujan type identities}

%%=============================================================%%
%% GivenName	-> \fnm{Joergen W.}
%% Particle	-> \spfx{van der} -> surname prefix
%% FamilyName	-> \sur{Ploeg}
%% Suffix	-> \sfx{IV}
%% \author*[1,2]{\fnm{Joergen W.} \spfx{van der} \sur{Ploeg} 
%%  \sfx{IV}}\email{iauthor@gmail.com}
%%=============================================================%%

\author[1]{\sur{Xiangxin Liu}}\email{liuxx@mail.nankai.edu.cn}

\author*[2]{\sur{Lisa Hui Sun}}\email{sunhui@nankai.edu.cn}
%\equalcont{These authors contributed equally to this work.}

\affil*[1,2]{Center for Combinatorics, LPMC, Nankai University, Tianjin 300071, P.R. China}

%%==================================%%
%% Sample for unstructured abstract %%
%%==================================%%

\abstract{Rogers--Ramanujan type identities occur in various branches of mathematics and physics. As a classic and powerful tool to deal with Rogers--Ramanujan type identities, the theory of Bailey's lemma has been extensively studied and generalized. In this paper, we found a bilateral Bailey pair that naturally arises from the $q$-binomial theorem. By applying the bilateral versions of Bailey lemmas, Bailey chains and Bailey lattices, we derive a number of Rogers--Ramanujan type identities, which unify many known identities as special cases. Further combined with the bilateral Bailey chains due to Berkovich, McCoy and Schilling and the bilateral Bailey lattices due to Jouhet et al., we also obtain identities on Appell--Lerch series and identities of Andrews--Gordon type. Moreover, by applying Andrews and Warnaar's bilateral Bailey lemmas, we derive identities on   Hecke--type series.}


%%\pacs[JEL Classification]{D8, H51}


\maketitle

\section{Introduction}

The famous Rogers--Ramanujan identities
\[
\sum_{n=0}^{\infty}\frac{q^{n^2}}{(q; q)_n}=\frac{1}{(q, q^4; q^5)_\infty},\quad
\sum_{n=0}^{\infty}\frac{q^{n^2+n}}{(q; q)_n}=\frac{1}{(q^2, q^3; q^5)_\infty},
\]
are due to Rogers \cite{Rog02} in 1894 and were rediscovered independently by Ramanujan \cite{RR}. Rogers and many other researchers discovered many series--product identities similar in the form of the Rogers--Ramanujan identities, and thereby are called ``Rogers--Ramanujan type identities".  In 1949, inspired by Rogers' proof of the Rogers--Ramanujan identities \cite{Rog01}, Bailey published a remarkable article \cite{Bai49} and obtained a number of Rogers--Ramanujan type identities by using the mechanism what is called the Bailey transform. The key of the Bailey transform is the construction of the Bailey pair, which is given as a pair of sequences of rational functions $(\alpha_n(a,q), \beta_n(a,q))_{n\geq 0}$  with respect to $a$ such that
\begin{align}\label{Baileyp}
\beta_n(a,q)=\sum_{j=0}^n \frac{\alpha_j(a,q)}{(q)_{n-j}(aq)_{n+j}}.
\end{align}
For a Bailey pair in the form of  \eqref{Baileyp}, Bailey \cite{Bai49} provided  the fundamental result named Bailey's lemma, which is formalized in the following lemma due to Andrews \cite{And842},  who laid the foundations of the Bailey pair and Bailey chain theory for discovering and proving $q$-identities.

\begin{lemma}[Bailey's lemma] \label{baileylem} If $\alpha_n(a,q), \beta_n(a,q)$ form a Bailey pair, then
\begin{align*}
&\frac{1}{(aq/\rho_1,aq/\rho_2)_n}\sum_{j=0}^n \frac{(\rho_1,\rho_2)_j(aq/\rho_1\rho_2)_{n-j}}{(q)_{n-j}} \bigg(\frac{aq}{\rho_1\rho_2}\bigg)^j \beta_j(a,q)\\
&\qquad =\sum_{j=0}^n \frac{(\rho_1,\rho_2)_j}{(q)_{n-j}(aq)_{n+j}(aq/\rho_1,aq/\rho_2)_j} \bigg(\frac{aq}{\rho_1\rho_2}\bigg)^j \alpha_j(a,q).
\end{align*}
\end{lemma}

One of the most famous result in the theory of Bailey's lemma is Slater's list, which is constructed by Slater \cite{Sla51,Sla52} based on the theory of basic hypergeometric series and leads to more than 130  Rogers--Ramanujan type identities. Then Laughlin,  Sills, Zimmer, Lovejoy, et. al.  \cite{And03,And04,Sil02, Sil03, Lov041,Lov13, Sun01,Yao01} further  collated and derived more Rogers--Ramanujan type identities.

MacMahon\cite{MacMahon16}  and Schur \cite{SchurRR} also
gave the combinatorial versions of the two Rogers--Ramanujan identities. In 1961, Gordon \cite{Go} proved  the stronger combinational results for Rogers--Ramanujan  identities. Then in 1974, Andrews \cite{A74} expressed Gordon's result in terms of the following $q$-identity
\begin{align}\label{eq:AGri}
\sum_{n_1\geq\dots\geq n_{k-1}\geq0}\frac{q^{n_1^2+\cdots+n_{k-1}^2+n_{i}+\cdots+n_{k-1}}}
{(q)_{n_1-n_2}\cdots(q)_{n_{k-2}-n_{k-1}}(q)_{n_{k-1}}}
=\frac{(q^{i},q^{2k-i+1},q^{2k+1};q^{2k+1})_\infty}{(q)_\infty},
\end{align}
with $k \geq 2$ and $1 \leq i \leq k$ being two integers, which is  the well-known Andrews--Gordon identity. To prove this identity, a more general tool, the Bailey lattice, was developed by Agarwal and Andrews in \cite{AA}, which can iteratively generate new Bailey pairs with respect to $a/q$ from a known one with respected to $a$, the related  results  can be seen in \cite{AA,BIS00,Lov041}. Then in 2010, the Andrews--Gordon identity was further extended by Andrews \cite{And10} to the following even modulo form,
\begin{align}\label{eq:AGev}
&\sum_{n_1\geq\cdots\geq n_{k-1}\geq0}\frac{q^{n_1^2+\cdots+n_{k-1}^2+2(n_{i}+n_{i+2}+\cdots+n_{k-2})}}
{(q^2;q^2)_{n_1-n_2}\cdots(q^2;q^2)_{n_{k-2}-n_{k-1}}(q^2;q^2)_{n_{k-1}}}\nonumber\\
&\qquad\qquad\qquad\qquad\qquad\qquad\qquad=\frac{(-q;q^2)_\infty
(q^{i},q^{2k-i+2},q^{2(k+1)};q^{2(k+1)})_\infty}{(q^2;q^2)_\infty},
\end{align}
where  $1\leq i \leq k$ and $k\equiv i \pmod{2}$ are two integers.

Analogous to the definition of Bailey pair, in 1987, Paule \cite{Pau87} had studied the bilateral versions of Bailey pairs. In order to study the character polynomials of the module $M(p, p+1)$, Berkovich, McCoy and Schilling \cite{BMS96} introduced the definition of the full bilateral version of Bailey pairs  in 1996, that is, for any $n\in\mathbb{Z}$, a bilateral Bailey pair $(\alpha_n(a,q),\,\beta_n(a,q))$ relative to $a$ is defined to be a pair of sequences of rational functions such that
\begin{align}\label{bbp}
\beta_n(a,q)=\sum_{j=-\infty}^{n}\frac{\alpha_j(a,q)}{(q)_{n-j}(aq)_{n+j}}.
\end{align}
By taking $a=1$, we can see that \eqref{bbp} can be written as
\begin{align}\label{basicbbf}
\beta_n(1,q)=\begin{cases} \sum\limits_{j=-n}^{n}\frac{\alpha_j(1,q)}{(q)_{n-j}(q)_{n+j}}, & \hbox{if $n\geq 0$;} \\[2mm]
0, & \hbox{otherwise}.
\end{cases}
\end{align}

In 2007, Andrews and Warnaar \cite{AW02} gave a bilateral version of the Bailey transform. Then in 2009, Guo, Jouhet and Zeng \cite{GZJ01} studied  the finite forms of the Rogers--Ramanujan identities by applying the bilateral Bailey lemma. Encouraged by these works, Chu and Zhang \cite{Chu01} further constructed a bilateral Bailey lemma related to $a=1$.

In this paper, we found the following bilateral Bailey pair naturally from the $q$-binomial theorem
\begin{align}
(\alpha_n(1,q),\ \beta_n(1,q))=\Big((-1)^nz^nq^{\binom{n}{2}}, \ \frac{(z, q/z)_n}{(q)_{2n}} \Big).\label{mainbbp}
\end{align}
By fitting the bilateral Bailey pair into the bilateral Bailey lemmas, we are led to a number of Rogers--Ramanujan type identities, such as
\begin{align}\label{RRmod21}
&\sum_{n=0}^{\infty}\frac{(q, q^6; q^7)_n}{(q^7; q^7)_{2n}}q^{7n^2}=\frac{(q^8, q^{13}, q^{21}; q^{21})_\infty}{(q^7; q^7)_{\infty}},\\
&\sum_{n=0}^{\infty}\frac{(-q^3, -q^5; q^8)_n}{(q^8; q^8)_{2n}}q^{8n^2}=\frac{(-q^{11}, -q^{13}, q^{24}; q^{24})_\infty}{(q^8; q^8)_{\infty}}. \label{RRmod24}
\end{align}
Moreover, the key bilateral Bailey pair \eqref{mainbbp} also leads to  identities on Appell--Lerch series and identities of  Andrews--Gordon type, which will be shown in details in Section 3 and Section 4.   In Section 5, by applying  Andrews and Warnaar's bilateral Bailey lemmas, we obtain identities on Hecke--type series.

\section{Preliminaries}
Throughout this paper, we adopt  standard notations and terminologies
for $q$-series \cite{GR04}. We assume that $|q|<1$, then the $q$-shifted factorial is given by
\[
(a)_n=(a;q)_n=\begin{cases}
1, & \text{\it if $n=0$}; \\[2mm]
\prod\limits_{j=1}^{n}(1-aq^{j-1}), & \text{\it if $n\geq 1$},
%; \\[2mm]
%\prod\limits_{j=1}^{-n}(1-aq^{-j})^{-1}, & \text{\it if $n\leq -1$},
\end{cases}
\]
and
\[
(a)_\infty=(a;q)_\infty=\prod_{n=0}^\infty (1-aq^n).
\]
There are more compact notations for the multiple $q$-shifted factorials:
\begin{align*}
&(a_1,a_2,\cdots,a_m)_n=(a_1,a_2,\cdots,a_m;q)_n=(a_1;q)_n(a_2;q)_n \cdots(a_m;q)_n,\\
&(a_1,a_2,\cdots,a_m)_{\infty}=(a_1,a_2,\cdots,a_m;q)_{\infty}=(a_1;q)_{\infty}
(a_2;q)_{\infty}\cdots(a_m;q)_{\infty}.
\end{align*}
The $q$-binomial coefficients, or Gaussian polynomials are given by
${n \brack k}=\frac{(q)_n}{(q)_k(q)_{n-k}}$. 
%we also denote the case when $q\to q^\ell$ by ${n\brack k}_{q^{\ell}}$.
The $q$-binomial theorem \cite[(3.3.6)]{And00}, also called Cauchy's binomial theorem, is known as
\begin{align}
(z)_n=\sum^{n}_{j=0}{n\brack j}(-1)^j q^{\binom{j}{2}} z^j .\label{q-binomial-thm}
\end{align}
The well-known Jacobi's triple product identity \cite[(II.28)]{GR04} is
\begin{align}
(z, q/z, q; q)_\infty=\sum_{n=-\infty}^{\infty}(-1)^nq^{\binom{n}{2}}z^n. \label{JTP}
\end{align}

In \cite{BMS96},  Berkovich, McCoy and Schilling gave a bilateral Bailey chain as follows.

\begin{lemma}\label{thm:bilatbaileylemma}
If $(\alpha_n,\,\beta_n)$ is a bilateral Bailey pair related to $a$, then
so is $(\alpha'_n,\,\beta'_n)$, where
\begin{align*}
&\alpha'_n(a,q)=\frac{(x,y)_n}{(aq/x,aq/y)_n}\left(\frac{aq}{xy}\right)^n\alpha_n(a,q),\\
&\beta'_n(a,q)=\frac{1}{(aq/x,aq/y)_n}\sum_{j=-\infty}^n\frac{(x,y)_j(aq/xy)_{n-j}}{(q)_{n-j}}\left(\frac{aq}{xy}\right)^j\beta_j(a,q),
\end{align*}
subject to convergence conditions on the sequences $\alpha_n$ and $\beta_n$, which make the relevant infinite series absolutely convergent.
\end{lemma}
By inserting the bilateral Bailey pair $(\alpha'_n(a,q),\,\beta'_n(a,q))$ as given in Lemma \ref{thm:bilatbaileylemma} into the definition of bilateral Bailey pairs \eqref{bbp} and letting $n\rightarrow\infty$, Berkovich, McCoy and Schilling \cite{BMS96} also gave the following bilateral Bailey lemma.

\begin{lemma}[BMS' bilateral Bailey lemma] \label{bbpinfinite}
If $(\alpha_n,\,\beta_n)$ is a bilateral Bailey pair related to $a$, then
\begin{align*}
&\sum_{n=-\infty}^{\infty} (x, y)_{n}\left(\frac{aq}{xy}\right)^{n}\beta_{n}(a,q)=\frac{(aq/x,aq/y)_{\infty}}{(aq,aq/xy)_{\infty}}\sum_{n=-\infty}^{\infty} \frac{(x,y)_n}{(aq/x,aq/y)_{n}} \left(\frac{aq}{x y}\right)^{n}\alpha_{n}(a,q).
\end{align*}
\end{lemma}

Furthermore, when $a=1$ and by noting  \eqref{basicbbf}, the above result reduces to the following weak forms.

\begin{lemma}\label{weakbbailey} For a bilateral Bailey pair $(\alpha_n (1,q), \beta_n(1,q))$, we have
\begin{subequations}
\begin{align}
&\sum_{n=0}^{\infty}q^{n^2}\beta_n=\frac{1}{(q)_\infty}\sum_{n=-\infty}^{\infty}q^{n^2}\alpha_n, \label{bbaileyiden1}\\
&\sum_{n=0}^{\infty}q^{\frac{1}{2}n^2}(-q^{\frac{1}{2}})_n\beta_n
=\frac{(-q^{\frac{1}{2}})_\infty}{(q)_\infty}\sum_{n=-\infty}^{\infty}
q^{\frac{1}{2}n^2}\alpha_n,\label{bbaileyiden4}\\
&2\sum_{n=0}^{\infty}(-1)^n(q; q^2)_n\beta_n
=\frac{(q; q^2)_\infty}{(q^2; q^2)_\infty}\sum_{n=-\infty}^{\infty}
(-1)^n\alpha_n,\label{bbaileyiden5}\\
&\sum_{n=0}^{\infty}q^{\binom{n}{2}}(-q)_n\beta_n
=\frac{(-q)_\infty}{(q)_\infty}\sum_{n=-\infty}^{\infty}
(1+q^n)q^{\binom{n}{2}}\alpha_n,\label{bbaileyiden3}\\
&\sum_{n=0}^{\infty}q^{\binom{n+1}{2}}(-1)_n\beta_n
=\frac{2(-q)_\infty}{(q)_\infty}\sum_{n=-\infty}^{\infty}
\frac{q^{\binom{n+1}{2}}}{1+q^n}\alpha_n.\label{bbaileyiden2}
\end{align}
\end{subequations}
\end{lemma}

\begin{proof}
These results can be obtained by taking $a=1$, and then $x\rightarrow\infty, y\rightarrow\infty$; $x\rightarrow\infty, y\mapsto -\sqrt{q}$; $x\mapsto \sqrt{q}, y\mapsto -\sqrt{q}$; $x\rightarrow\infty, y\mapsto -q$; $x\rightarrow\infty, y\mapsto -1$, respectively in Lemma \ref{bbpinfinite}.
\end{proof}

Moreover, for a given  bilateral Bailey pair $(\alpha_{n}(1,q), \beta_{n}(1,q))$, employing the bilateral Bailey chain as given in Lemma \ref{thm:bilatbaileylemma} iteratively, we can obtain a sequence of bilateral Bailey pairs, which lead to  the following identities on multisums.

\begin{lemma}\label{kbblit} If
$(\alpha_{n},\ \beta_{n})$ is a bilateral Bailey pair relative to $a=1$, then
\begin{subequations}
\begin{align}
&\sum_{n_1\geq\cdots\geq n_k\geq 0}\frac{q^{n_1^2+\cdots+n_k^2}}{(q)_{n_1-n_2}\cdots (q)_{n_{k-1}-n_k}}\beta_{n_k}
=\frac{1}{(q)_\infty}\sum_{n=-\infty}^{\infty}q^{kn^2}\alpha_n,\label{multisum1}\\
&\sum_{n_1\geq\cdots\geq n_k\geq 0}\frac{q^{\frac{1}{2}(n_1^2+n_2^2\cdots+n_k^2)}(-q^{\frac{1}{2}})_{n_k}}
{(q)_{n_1-n_2}\cdots (q)_{n_{k-1}-n_{k}}}\beta_{n_k}
=\frac{(-q^{\frac{1}{2}})_\infty}{(q)_\infty}\sum_{n=-\infty}^{\infty}
q^{\frac{k}{2}n^2}\alpha_n.\label{multisum2}
\end{align}
\end{subequations}
\end{lemma}

\begin{proof} When $a=1$, $x\rightarrow\infty, y\rightarrow\infty$ and by noting \eqref{basicbbf}, the bilateral Bailey chain as given in Lemma \ref{thm:bilatbaileylemma} can be represented as
\begin{align*}
&\alpha^{(1)}_n=q^{n^2} \alpha_n,\quad \
\beta^{(1)}_n=\sum_{j=0}^{n} \frac{q^{{j}^2}}{(q)_{n-j}}\beta_{j}.
\end{align*}
Iteratively the above relation $k-1$ times, we obtain the following bilateral Bailey pair
\begin{align*}
&\alpha^{(k-1)}_n=q^{(k-1)n^2} \alpha_n,\\
&\beta^{(k-1)}_n=\sum_{n_{k-1}=0}^{n}\sum_{n_{k-2}=0}^{n_{k-1}}\sum_{n_{k-3}=0}^{n_{k-2}} \cdots \sum_{n_{1}=0}^{n_{2}} \frac{q^{n_1^2+n_2^2+\cdots +n_{k-1}^2}}{(q)_{n-n_{k-1}(q)_{n_{k-1}-n_{k-2}}\cdots (q)_{n_2-n_1}}}\beta_{n_1}.
\end{align*}
Substituting the above bilateral Bailey pair into \eqref{bbaileyiden1}, we have that
\[
\sum_{n\geq 0} \sum_{n_{k-1}=0}^{n}\sum_{n_{k-2}=0}^{n_{k-1}}\cdots \sum_{n_{1}=0}^{n_{2}} \frac{q^{n^2+n_1^2+n_2^2+\cdots +n_{k-1}^2}}{(q)_{n-n_{k-1}(q)_{n_{k-1}-n_{k-2}}\cdots (q)_{n_2-n_1}}}\beta_{n_1} = \frac{1}{(q)_\infty}\sum_{n=-\infty}^\infty q^{kn^2}\alpha_n.
\]
Then by replacing the variables on the left hand side with $n\mapsto n_1$ and $n_j\mapsto n_{k+1-j}$ for $1\leq j \leq k-1$, we arrive at \eqref{multisum1}. Simliarly, \eqref{multisum2} can be proved by applying the weak form of bilateral Bailey lemma \eqref{bbaileyiden4}.
\end{proof}

In 2000, Bressoud, Ismail and Stanton \cite{BIS00} gave a  Bailey lattice for the classic Bailey pairs. Then in \cite{JFI11} and \cite{Jou10}, the authors showed that the Bailey  lattice in \cite{BIS00} can be extended to  the following bilateral versions.

\begin{lemma}{\rm \cite[(2.10)]{JFI11}}\label{wnewbbf1}
If $(\alpha_n,\,\beta_n)$ is a bilateral Bailey pair related to $a$, then so is $(\alpha'_n,\,\beta'_n)$, where
\begin{align*}
&\alpha'_n(a,q)=\frac{1+a}{1+aq^{2n}}q^n\alpha_n(a^2,q^2),\\
&\beta'_n(a,q)=\sum_{j=-\infty}^n\frac{(-a)_{2j}q^j}{(q^2;q^2)_{n-j}}\beta_j(a^2,q^2).
\end{align*}
\end{lemma}

\begin{lemma}{\rm \cite[(3.2)]{Jou10}}\label{wnewbbf2}
If $(\alpha_n,\beta_n)$ is a bilateral Bailey pair related to $a$, then so is $(\alpha'_n,\,\beta'_n)$, where
\begin{align*}
\alpha'_n(a,q)&=\alpha_n(a^2,q^2),\\
\beta'_n(a,q)&=\sum_{j=-\infty}^n\frac{(-aq)_{2j}q^{n-j}}{(q^2;q^2)_{n-j}}\beta_j(a^2,q^2).
\end{align*}
\end{lemma}

\section{The key bilateral Bailey pair}

In this section, we construct our key bilateral Bailey pair and by inserting it into different forms of the bilateral Bailey lemmas, we obtain a number of Rogers--Ramanujan type identities, identities on Appell--Lerch series and identities of  Andrews--Gordon type.

Firstly, let us  give the key bilateral Bailey pair.

\begin{lemma}\label{basicbbp}
We have that
\begin{align}
(\alpha_n(1,q),\ \beta_n(1,q))=\Big((-1)^nz^nq^{\binom{n}{2}}, \ \frac{(z, q/z)_n}{(q)_{2n}} \Big)\label{originalbbp}
\end{align}
forms  a  bilateral Bailey pair with respective to $a=1$.
\end{lemma}

\begin{proof} Consider the product of two $q$-shifted factorials $(z,q/z)_n$, it can be rewritten as follows
\begin{align*}
(z,q/z)_n &=(z)_n(-1)^nz^{-n}q^{\binom{n+1}{2}}(zq^{-n})_n\\
&=(-1)^nz^{-n}q^{\binom{n+1}{2}}(zq^{-n})_{2n}.
\end{align*}
By applying the $q$-binomial theorem \eqref{q-binomial-thm}, the above expression can be expanded as follows
\begin{align*}
(-1)^nz^{-n}q^{\binom{n+1}{2}}\sum_{j=0}^{2n}{2n \brack j}(-1)^jq^{\binom{j}{2}}(zq^{-n})^j.
\end{align*}
Letting $j\mapsto n+j$ and simplifying, it turns to be
\begin{align}
(z,q/z)_n=\sum_{j=-n}^n {2n \brack n+j} (-1)^jq^{\binom{j}{2}}z^j,\label{basicB}
\end{align}
which completes the proof by noting the definition of bilateral Bailey pair \eqref{bbp} with $a=1$.
\end{proof}

In this paper, we will show how this key bilateral Bailey pair can be used in deriving new identities. Note that \eqref{basicB} can be seen as a truncated form  of Jacobi's triple product identity \eqref{JTP} in which by letting $n\rightarrow \infty$, it will lead to Jacobi's triple product identity.
Moreover, the bilateral Bailey pair \eqref{originalbbp} is also a special case of the Bailey pair given by \cite[P. 366]{Chu01} with $\delta=0$. If $z=1$ or $z=q$, it leads to a unit bilateral Bailey pair $\big((-1)^nq^{\binom{n}{2}}, \delta_{n,0}\big)$.

 By using the weak form \eqref{bbaileyiden1}   and with the help of Jacobi's triple product identity \eqref{JTP}, we  obtain the following Rogers--Ramanujan type identity directly.

\begin{theorem}\label{oriden}
We have
\begin{align}
&\sum_{n=0}^\infty\frac{(z, q/z)_n}{(q)_{2n}}q^{n^2}=\frac{(zq, q^2/z, q^3; q^3)_\infty}{(q)_\infty}.\label{weak1}
\end{align}
\end{theorem}

Note that by substituting $q\mapsto q^2, z\mapsto -aq$ into \eqref{weak1}, it turns to be Entry 5.3.1 as given in Ramanujan's ``Lost" Notebook \cite[P. 99]{AB09}.

It is interesting to see that by taking $q\mapsto q^m, z\mapsto\pm q^a$ in \eqref{weak1} respectively, we derive the following pair of identities.

\begin{corollary}\label{Bbpi1}
For $m, a\in\mathbb{N}$, $m\geq1, 0\leq a\leq m$, we have
\begin{align}
&\sum_{n=0}^{\infty}\frac{(q^a, q^{m-a}; q^m)_n}{(q^m; q^m)_{2n}}q^{mn^2}=\frac{(q^{m+a}, q^{2m-a}, q^{3m}; q^{3m})_\infty}{(q^m; q^m)_\infty},\label{bbpi4}\\
&\sum_{n=0}^{\infty}\frac{(-q^a, -q^{m-a}; q^m)_n}{(q^m; q^m)_{2n}}q^{mn^2}=\frac{(-q^{m+a}, -q^{2m-a}, q^{3m}; q^{3m})_\infty}{(q^m; q^m)_\infty}. \label{bbpi1}
\end{align}
\end{corollary}

When $m$ is even and $a$ is odd, it is obvious to see that \eqref{bbpi4} and \eqref{bbpi1} are equivalent.
It is notable that these two identities unify many known Rogers--Ramanujan type identities.
For example, when $m=2, a=1$ in \eqref{bbpi4}, it can be found in \cite[(4.3)]{Sun01} or Entry 5.3.3 in  \cite[P. 102]{AB09}
\begin{align*}
\sum_{n=0}^{\infty}\frac{(q; q^2)^2_n}{(q^2; q^2)_{2n}}q^{2n^{2}}=\frac{(q^3, q^3, q^6; q^6)_\infty}{(q^2; q^2)_\infty}.
\end{align*}
Moreover, in \eqref{bbpi4}, if we take $m=3, a=1$, it can be found in \cite[(42)]{Sla52}; when $m=4, a=1$, it also can be seen in \cite[(53)]{Sla52}; when $m=5, a=1$, it is the one in \cite[(2.15.5)]{Sil03}; when $m=5, a=2$, it in the form of \cite[(2.15.6)]{Sil03}; when $m=6, a=1$, it also can be found in \cite[(135)]{Chu01}. For \eqref{bbpi1}, if we take $m=1, a=0$, it in the form of \cite[(A.6)]{Sil02}; when $m=3, a=1$, it can be found in \cite[(86)]{Chu01}.

Besides, these two identities also lead to  more identities with modulo  $3m$, which may be added into the list of Rogers--Ramanujan type identities with modulo divisible by 3.  For examples, when $m=7$ and $a=1$ in \eqref{bbpi4}, it leads to \eqref{RRmod21}.
When $m=8$ and $a=3$ in \eqref{bbpi1}, it reduces to \eqref{RRmod24}.
When $m=12$ and  $a=5$ in \eqref{bbpi4}, it leads to the  Rogers--Ramanujan type identity
\begin{align*}
\sum_{n=0}^{\infty}\frac{(q^5, q^7; q^{12})_n}{(q^{12}; q^{12})_{2n}}q^{12n^{2}}=\frac{(q^{17}, q^{19}, q^{36}; q^{36})_\infty}{(q^{12}; q^{12})_\infty}.
\end{align*}
If we let $m=13, a=2$ in \eqref{bbpi1}, it gives
\begin{align*}
\sum_{n=0}^{\infty}\frac{(-q^2, -q^{11}; q^{13})_n}{(q^{13}; q^{13})_{2n}}q^{13n^{2}}=\frac{(-q^{15}, -q^{24}, q^{39}; q^{39})_\infty}{(q^{13}; q^{13})_\infty}.
\end{align*}

Next, by inserting the bilateral Bailey pair \eqref{originalbbp} into \eqref{bbaileyiden4}, we are lead to the following result.

\begin{theorem}
We have
\begin{align}
\sum_{n=0}^\infty\frac{(z, q/z, -q^{\frac{1}{2}})_n}{(q)_{2n}}q^{\frac{ n^2}{2}}=\frac{(-q^{\frac{1}{2}})_\infty
(zq^{\frac{1}{2}}, q^{\frac{3}{2}}/z, q^2; q^2)_\infty}
{(q)_\infty}.\label{weak4}
\end{align}
\end{theorem}

 Let $q\mapsto q^2, z\mapsto -aq$ in \eqref{weak4}, it implies Entry 5.3.5 as given in   \cite{AB09}.

When we take $q\mapsto q^m, z\mapsto\pm q^a$ in \eqref{weak4}, respectively, we obtain the following pair of identities.

\begin{corollary}\label{Bbpi125}
For $m, a\in\mathbb{N}$, $m\geq1, 0\leq a\leq m$, we have
\begin{align}
&\sum_{n=0}^{\infty}\frac{(q^a, q^{m-a}, -q^{\frac{m}{2}}; q^{m})_n}{(q^{m}; q^{m})_{2n}}q^{\frac{mn^2}{2}}
=\frac{(-q^{\frac{m}{2}}; q^{m})_\infty}{(q^{m}; q^{m})_\infty}\big(q^{\frac{m}{2}+a}, q^{\frac{3m}{2}-a}, q^{2m}; q^{2m}\big)_\infty ,\label{bbpi5-2}\\
&\sum_{n=0}^{\infty}\frac{(-q^a, -q^{m-a}, -q^{\frac{m}{2}}; q^{m})_n}{(q^{m}; q^{m})_{2n}}q^{\frac{mn^2}{2}}
=\frac{(-q^{\frac{m}{2}}; q^{m})_\infty}{(q^{m}; q^{m})_\infty}\big(-q^{\frac{m}{2}+a}, -q^{\frac{3m}{2}-a}, q^{2m}; q^{2m}\big)_\infty.\label{bbpi2-2}
\end{align}
\end{corollary}

Obviously, when $m=4k, k>0$, and $a$ is odd, \eqref{bbpi5-2} and \eqref{bbpi2-2} are equivalent. For \eqref{bbpi5-2}, if we set $m=2, a=1$, it can also be found in \cite[(4)]{Sla52} or \cite[(5.2a)]{Sun01}:
\begin{align*}
\sum_{n=0}^\infty\frac{(q; q^2)_n}{(q^4; q^4)_{n}}q^{n^2}=(-q; q^2)_\infty(q^2; q^4)_\infty.
\end{align*}
Furthermore, in \eqref{bbpi5-2}, when $m=4$ and $a=1$, it in the form of \cite[(2.8.12)]{Sil03}; if $m=6$ and $a=1$, it can be seen in \cite[(117)]{Chu01}; by taking $m=6$ and $a=2$, it turns to be \cite[(2.12.11)]{Sil03}. For \eqref{bbpi2-2}, when $m=2$ and $a=1$, it has the form of \cite[(5.2e)]{Sun01}; for $m=6$ and $a=1$, it can be found in \cite[(117)]{Chu01}; if $m=6$ and $a=2$, it is the identity \cite[(2.12.11)]{Sil03}.

These two identities as given in Corollary \ref{Bbpi125} also can imply more identities with modulo $2m$, which may be added into the list of Rogers--Ramanujan identities with even modulo. For example, when $m=8$ and $a=1$ in \eqref{bbpi5-2}, it gives the identity as follows
\begin{align*}
\sum_{n=0}^{\infty}\frac{(q,-q^4,q^7;q^{8})_n}{(q^{8};q^{8})_{2n}}q^{4n^2}=\frac{(-q^4; q^8)_\infty}{(q^8; q^8)_{\infty}}(q^5, q^{11}, q^{16}; q^{16})_\infty,
\end{align*}
when $m=10$ and $ a=4$ in \eqref{bbpi2-2}, we get
\begin{align*}
\sum_{n=0}^{\infty}\frac{(-q^4,-q^5,-q^6;q^{10})_n}{(q^{10};q^{10})_{2n}}q^{5n^2}=
\frac{(-q^5;q^{10})_\infty}{(q^{10};q^{10})_\infty}(-q^9,-q^{11},q^{20};q^{20})_\infty.
\end{align*}

Then by using the weak form  \eqref{bbaileyiden5}, we obtain the following identity.

\begin{theorem} We have
\begin{align}
2\sum_{n=0}^\infty(-1)^n\frac{(z, q/z)_n}{(q^2; q^2)_{n}}=\frac{(q; q^2)_\infty(-z, -q/z, q; q)_\infty}
{(q^2; q^2)_\infty}.\label{weak5}
\end{align}
\end{theorem}

Note  that by setting  $q\mapsto q^2, z\mapsto -aq$ in \eqref{weak5}, it just leads to the identity as given in \cite[Entry 5.3.10]{AB09}. When we take $q\mapsto q^2, z\mapsto -q$ in \eqref{weak5}, it reduces to \cite[(5.4c)]{Sun01}.

By inserting the bilateral Bailey pair \eqref{originalbbp} into \eqref{bbaileyiden3}, it gives the following three--term identity.

\begin{theorem} We have
\begin{align}
\sum_{n=0}^\infty\frac{(z, q/z, -q)_n}{(q)_{2n}}q^{\binom{n}{2}}=\frac{(-q)_\infty}
{(q)_\infty}\big((z, q^2/z, q^2; q^2)_\infty+(zq, q/z, q^2; q^2)_\infty\big).\label{weak3}
\end{align}
\end{theorem}

When $q\mapsto q^4, z\mapsto q$ in \eqref{weak3}, it can be found in \cite[ (3.8.1)]{Sil03}
\begin{align*}
\sum_{n=0}^\infty\frac{(-q^4; q^4)_{n}(q; q^2)_{2n}}{(q^4; q^4)_{2n}}q^{4\binom{n}{2}}=\frac{(-q^4; q^4)_\infty}{(q^4; q^4)_\infty}\left((q, q^7, q^8; q^8)_\infty+(q^3, q^5, q^8; q^8)_\infty\right).
\end{align*}

From \eqref{weak3}, we also can obtain more identities. For example, letting  $q\mapsto q^7, z\mapsto\pm q$ in \eqref{weak3}, respectively, it gives that

\begin{align*}
&\sum_{n=0}^{\infty} \frac{(\pm q, \pm q^6, -q^7; q^7)_{n}}{(q^7; q^7)_{2n}}q^{7\binom{n}{2}}\\
&\qquad\qquad=\frac{(-q^7; q^7)_\infty}{(q^7; q^7)_\infty}\big((\pm q, \pm q^{13}, q^{14}; q^{14})_\infty+(\pm q^6, \pm q^8, q^{14}; q^{14})_\infty\big).
\end{align*}

Furthermore, we see that by taking the bilateral Bailey pair \eqref{originalbbp} into \eqref{bbaileyiden2}, some identities on Appell--Lerch series can be  derived.
Recall that the Appell--Lerch series are of the following form
\begin{equation}\label{appler}
\sum_{n=-\infty}^\infty \frac{(-1)^{\ell n}q^{\ell n(n+1)/2}b^n}{1-aq^n},
\end{equation}
which was first studied by Appell \cite{App846} and Lerch \cite{Ler92}. After multiplying the series \eqref{appler} by the factor $a^{\ell/2}$ and viewing it as function in the variables $a$, $b$ and $q$, it is also refereed as an Appell function of level $\ell$.

\begin{theorem} We have
\begin{align}
\sum_{n=0}^\infty\frac{(-1, z, q/z)_n}{(q)_{2n}}q^{\binom{n+1}{2}}=\frac{2(-q)_\infty}
{(q)_\infty}\sum_{n=-\infty}^{\infty}\frac{(-1)^nz^nq^{n^2}}{1+q^n}.\label{weak2}
\end{align}
\end{theorem}

In fact, by taking $q\mapsto q^2, z\mapsto -q$ in \eqref{weak2}, it reduces to  \cite[(6.3)]{Sun01}
\begin{align*}
\sum_{n=0}^{\infty} \frac{(-1, -q, -q; q^2)_n}{(q^2; q^2)_{2n}}q^{n(n+1)}=\frac{2(-q^2; q^2)_\infty}{(q^2; q^2)_\infty}\sum_{n=-\infty}^{\infty}\frac{q^{n(2n+1)}}{1+q^{2n}}.
\end{align*}

Finally, by taking the bilateral Bailey pair \eqref{originalbbp} into Lemma \ref{kbblit}, we obtain the following two identities of  Andrews--Gordon type.

\begin{theorem}\label{iterate}
We have
\begin{align}
&\sum_{n_1\geq\cdots\geq n_k\geq 0}\frac{q^{n_1^2+\cdots+n_k^2}(z, q/z)_{n_k}}{(q)_{n_1-n_2}\cdots (q)_{n_{k-1}-n_k}(q)_{2n_k}}=\frac{(q^kz, q^{k+1}/z, q^{2k+1}; q^{2k+1})_\infty}{(q)_\infty},\label{AG1}\\
&\sum_{n_1\geq\cdots\geq n_k\geq 0}\frac{q^{n_1^2+n_2^2+\cdots+n_k^2}(-q, z,q^2/z;q^2)_{n_k}}{(q^2;q^2)_{n_1-n_2}\cdots (q^2;q^2)_{n_{k-1}-n_k}(q^2;q^2)_{2n_k}}\nonumber\\
&\qquad\qquad\qquad\qquad\qquad=\frac{(-q;q^2)_\infty(q^{k}z, q^{k+2}/z, q^{2(k+1)};q^{2(k+1)})_\infty}{(q^2;q^2)_\infty}.\label{AG2}
\end{align}
\end{theorem}

Note that by substituting $z=1$ or $z=q$ in \eqref{AG1}, it reduces to
\begin{align*}
\sum_{n_1\geq\cdots\geq n_{k-1}\geq 0}\frac{q^{n_1^2+\cdots+n_{k-1}^2}}{(q)_{n_1-n_2}\cdots (q)_{n_{k-2}-n_{k-1}}(q)_{n_{k-1}}}
=\frac{(q^k,q^{k+1},q^{2k+1};q^{2k+1})_\infty}{(q)_\infty},
\end{align*}
which is a special case of Andrews--Gordon identity \eqref{eq:AGri} with $i=k$ and also can be found in \cite[(2.6)]{ASW01}. If we take $z=q^{i-k}$ or $z=q^{k-i+1}$ in \eqref{AG1}, it gives
\begin{align*}
\sum_{n_1\geq\cdots\geq n_{k}\geq 0}\frac{(q^{i-k},q^{k-i+1})_{n_k}q^{n_1^2+\cdots+n_k^2}}{(q)_{n_1-n_2}\cdots (q)_{n_{k-1}-n_{k}}(q)_{2n_k}}
=\frac{(q^{i},q^{2k-i+1},q^{2k+1};q^{2k+1})_\infty}{(q)_\infty},
\end{align*}
which gives another expansion formula for the right hand side of Andrews--Gordon identity \eqref{eq:AGri}. Moreover, for identity \eqref{AG1}, it can reduce to Theorem \ref{oriden} when $k=1$, and it also leads to \cite[(4.9c)]{Sun01} when $q\mapsto q^2$ and $z\mapsto q$.

For identity \eqref{AG2}, when $k=1, q\mapsto q^{\frac{1}{2}}$, it reduces to \eqref{weak4}. Letting $z\mapsto 1$ or $z \mapsto q^2$, we obtain the special case of \eqref{eq:AGev} with $i=k$.
When  $z\mapsto q^{i-k}$ or $z\mapsto q^{k+2-i}$ with $1\leq i \leq k+1$, we get the following Andrews--Gordon type identity
\begin{align*}
&\sum_{n_1\geq\cdots\geq n_k\geq 0}\frac{q^{n_1^2+n_2^2+\cdots+n_k^2}(q^{i-k},q^{k+2-i},-q;q^2)_{n_k}}
{(q^2;q^2)_{n_1-n_2}\cdots(q^2;q^2)_{n_{k-1}-n_k}(q^2;q^2)_{2n_k}}\\
&\qquad\qquad\qquad\qquad\qquad\qquad\qquad=\frac{(-q;q^2)_\infty
(q^i,q^{2k-i+2},q^{2(k+1)};q^{2(k+1)})_\infty}{(q^2;q^2)_\infty},
\end{align*}
which extends the identity \eqref{eq:AGev} of Andrews to be without the restriction $i \equiv k \pmod{2}$.

\section{The bilateral Bailey lattices}

In this section, by applying the bilateral Bailey pair \eqref{originalbbp} combined with the bilateral Bailey lattices as given in Lemma \ref{wnewbbf1} due to Dousse, Jouhet and  Konan \cite{JFI11}, and Lemma \ref{wnewbbf2} due to  Jouhet \cite{Jou10}, we obtain identities on Appell--Lerch series and also identities of  Rogers--Ramanujan type.

\begin{theorem}
We have
\begin{align}
\sum_{n=0}^\infty\frac{(z,q^2/z;q^2)_n}{(1+q^{2n})(q)_{2n}}q^n =\frac{(-q)_\infty}{(q)_\infty}
\sum_{n=-\infty}^\infty\frac{(-1)^nz^nq^{n^2}}{1+q^{2n}}.
\end{align}
\end{theorem}

\begin{proof}
First, setting $q\mapsto q^2$ in \eqref{originalbbp}, and substituting it into the bilateral Bailey lattice as given in Lemma \ref{wnewbbf1} with $a=1$, we obtain the following  bilateral Bailey pair
\begin{align}
(\alpha_n(1,q),\ \beta_n(1,q))=\Big(\frac{2(-1)^nz^nq^{n^2}}{1+q^{2n}},\ \sum_{j=0}^n\frac{(-1)_{2j}(z,q^2/z;q^2)_j}{(q^2;q^2)_{n-j}(q^2;q^2)_{2j}}q^j\Big).\label{newbbaileypair1}
\end{align}
Then inserting the above bilateral Bailey pair into  \eqref{basicbbf} with $a=1$, and letting  $n\rightarrow\infty$, we complete the proof.
\end{proof}

By employing  the bilateral Bailey lattice in Lemma \ref{wnewbbf1}, we derive the following result.

\begin{theorem}\label{itelattice1}
For a bilateral Bailey pair $(\alpha_n(a,q),\ \beta_n(a,q))$ and $k\geq 2$, we have
\begin{align}
&\sum_{n_1\geq\cdots\geq n_k\geq -\infty}(x,y)_{n_1}\Big(\frac{aq}{xy}\Big)^{n_1}\prod_{i=0}^{k-2}
\frac{(-a^{2^i};q^{2^i})_{2n_{i+2}}}
{(q^{2^{i+1}};q^{2^{i+1}})_{n_{i+1}-n_{i+2}}}q^{2^in_{i+2}}
\beta_{n_k}(a^{2^{k-1}},q^{2^{k-1}})\\
&\quad =\frac{(aq/x,aq/y)_{\infty}}{(aq,aq/xy)_{\infty}}\sum_{n=-\infty}^{\infty} \frac{(x,y)_n}{(aq/x,aq/y)_{n}} \Big(\frac{aq}{x y}\Big)^{n}\prod_{i=0}^{k-2}\frac{1+a^{2^{i}}}{1+(aq^{2n})^{2^{i}}}q^{2^in}
\alpha_{n}(a^{2^{k-1}},q^{2^{k-1}}). \nonumber
\end{align}
\end{theorem}

\begin{proof} Firstly, for a given bilateral Bailey pair $(\alpha_{n}(a,q),\ \beta_{n}(a,q))$, we apply the bilateral Bailey lattice as given in Lemma \ref{wnewbbf1} $k-1$ times with $k\geq 2$ iteratively, it gives the following bilateral Bailey pair
\begin{align*}
&\alpha_{n}^{(k-1)}(a,q)=\prod_{i=0}^{k-2}\frac{1+a^{2^{i}}}{1+(aq^{2n})^{2^{i}}}q^{2^in}
\alpha_{n}(a^{2^{k-1}},q^{2^{k-1}}),\\
&\beta_{n}^{(k-1)}(a,q)=\sum_{n\geq n_{k-1}\geq\cdots\geq n_1\geq -\infty}\prod_{i=1}^{k-1}
\frac{(-a^{2^{k-i-1}};q^{2^{k-i-1}})_{2n_i}}
{(q^{2^{k-i}};q^{2^{k-i}})_{n_{i+1}-n_i}}q^{2^{k-i-1}n_i}
\beta_{n_1}(a^{2^{k-1}},q^{2^{k-1}}),
\end{align*}
where $n_{k}=n$. Next, by making the parameter replacement $ n_i\mapsto n_{k-i+1}$ for $2\leq i \leq k-1$, and then taking $i\mapsto k-1-i$ in $\beta_{n}^{(k-1)}(a,q)$, and putting the resulting bilateral Bailey pair into Lemma \ref{bbpinfinite}.
\end{proof}

From the above theorem, we can derive some identities on Appell--Lerch series in the form of multisums. Now, we take $a=1, k=2$ in Theorem \ref{itelattice1} as an illustration. By letting $x\rightarrow\infty, y\rightarrow\infty$; $x\rightarrow\infty, y\mapsto -\sqrt{q}$; $x\mapsto \sqrt{q}, y\mapsto -\sqrt{q}$, respectively, and inserting into the main bilateral Bailey pair \eqref{originalbbp}, we obtain the following identities.

\begin{corollary}
We have
\begin{align}
&\sum_{n_1\geq n_2\geq 0}\frac{(z,q^2/z;q^2)_{n_2}q^{n_1^2+n_2}}
{(1+q^{2n_2})(q^2; q^2)_{n_1-n_2}(q)_{2n_2}}
=\frac{1}{(q)_\infty}\sum_{n=-\infty}^{\infty}
\frac{(-1)^nz^nq^{2n^2}}{1+q^{2n}},\label{Appell1}\\
&\sum_{n_1\geq n_2\geq 0}\frac{(-q^{\frac{1}{2}})_{n_1}(z,q^2/z;q^2)_{n_2}
q^{\frac{1}{2}n_1^2+n_2}}{(1+q^{2n_2})(q^2; q^2)_{n_1-n_2}(q)_{2n_2}}
=\frac{(-q^{\frac{1}{2}})_\infty}
{(q)_\infty}\sum_{n=-\infty}^{\infty}\frac{(-1)^nz^nq^{\frac{3}{2}n^2}}{1+q^{2n}},\\
&2\sum_{n_1\geq n_2\geq 0}\frac{(-1)^{n_1}(q;q^2)_{n_1}(z,q^2/z;q^2)_{n_2}q^{n_2}}
{(1+q^{2n_2})(q^2; q^2)_{n_1-n_2}(q)_{2n_2}}
=\frac{(q;q^2)_\infty}{(q^2;q^2)_\infty}\sum_{n=-\infty}^{\infty}\frac{z^nq^{n^2}}{1+q^{2n}}.
\end{align}
\end{corollary}

By comparing \eqref{weak2} with \eqref{Appell1}, it follows that
\begin{align*}
\sum_{n_1\geq n_2\geq 0}\frac{(z,q^2/z;q^2)_{n_2}q^{n_1^2+n_2}}
{(1+q^{2n_2})(q^2; q^2)_{n_1-n_2}(q)_{2n_2}}=\frac{(-q;q^2)_\infty}{2}\sum_{n=0}^\infty\frac{(-1, z, q^2/z;q^2)_n}{(q^2;q^2)_{2n}}q^{n(n+1)},
\end{align*}
which also can be proved by using Euler's $q$-exponential formula \cite[(II.2)]{GR04}. 

Similarly with the procedures as given in the proof of Theorem \ref{itelattice1} and by using the bilateral Bailey lattice as given in Lemma \ref{wnewbbf2}, we obtain the following result.

\begin{theorem}\label{itelattice2}
For a bilateral Bailey pair $(\alpha_n(a,q),\ \beta_n(a,q))$ and $k\geq 2$, we have
\begin{align}
&\sum_{n_1\geq\cdots\geq n_k\geq -\infty}(x,y)_{n_1}\Big(\frac{aq}{xy}\Big)^{n_1}\prod_{i=0}^{k-2}
\frac{\big(-(aq)^{2^i};q^{2^i}\big)_{2n_{i+2}}}
{(q^{2^{i+1}};q^{2^{i+1}})_{n_{i+1}-n_{i+2}}}q^{2^i(n_{i+1}-n_{i+2})}
\beta_{n_k}(a^{2^{k-1}},q^{2^{k-1}})\nonumber\\
&\qquad=\frac{(aq/x,aq/y)_{\infty}}{(aq,aq/xy)_{\infty}}\sum_{n=-\infty}^{\infty} \frac{(x,y)_n}{(aq/x,aq/y)_{n}} \Big(\frac{aq}{xy}\Big)^{n}
\alpha_{n}(a^{2^{k-1}},q^{2^{k-1}}).
\end{align}
\end{theorem}

Letting $a=1, k=2$ and taking $x\rightarrow\infty, y\rightarrow\infty$; $x\rightarrow\infty, y\mapsto -\sqrt{q}$; $x\mapsto \sqrt{q}, y\mapsto -\sqrt{q}$; $x\rightarrow\infty, y\mapsto -q$; $x\rightarrow\infty, y\mapsto -1$, respectively,  we deduce the following double sum identities of Rogers--Ramanujan type.

\begin{corollary}
We have
\begin{align}
&\sum_{n_1\geq n_2\geq 0}\frac{(z,q^2/z;q^2)_{n_2}}{(q^2; q^2)_{n_1-n_2}(q)_{2n_2}}q^{n_1^2+n_1-n_2}
=\frac{(zq,q^3/z,q^4;q^4)_\infty}{(q)_\infty},\label{new1}\\
&\sum_{n_1\geq n_2\geq 0}\frac{(-q^{\frac{1}{2}})_{n_1}(z,q^2/z;q^2)_{n_2}}{(q^2; q^2)_{n_1-n_2}(q)_{2n_2}}q^{\frac{1}{2}n_1^2+n_1-n_2}
=\frac{(-q^{\frac{1}{2}})_\infty(zq^{\frac{1}{2}},
q^{\frac{5}{2}}/z,q^3;q^3)_\infty}{(q)_\infty},\label{new2}\\
&2\sum_{n_1\geq n_2\geq 0}\frac{(-1)^{n_1}(q;q^2)_{n_1}(z,q^2/z;q^2)_{n_2}}{(q^2; q^2)_{n_1-n_2}(q)_{2n_2}}q^{n_1-n_2}
=(-z,-q^2/z,q;q^2)_\infty, \label{new3}\\
&\sum_{n_1\geq n_2\geq 0}\frac{(-q)_{n_1}(z,q^2/z;q^2)_{n_2}}{(q^2; q^2)_{n_1-n_2}(q)_{2n_2}}q^{\binom{n_1+1}{2}-n_2} \label{new4}\\
&\qquad\qquad\qquad\qquad=\frac{(-q)_\infty}{(q)_\infty}\big((z,q^3/z,q^3;q^3)_\infty+(zq, q^2/z,q^3;q^3)_\infty \big),\nonumber\\
&\sum_{n_1\geq n_2\geq 0}\frac{(-1)_{n_1}(z,q^2/z;q^2)_{n_2}}{(q^2; q^2)_{n_1-n_2}(q)_{2n_2}}q^{\binom{n_1+1}{2}+n_1-n_2} =\frac{2(-q)_\infty}{(q)_\infty}\sum_{n=-\infty}^\infty
\frac{(-1)^nz^nq^{n(3n-1)/2}}{1+q^n}.\label{new5}
\end{align}
\end{corollary}

Note that identity \eqref{new1} can be gotten by using  Euler's $q$-exponential formula \cite[(II.2)]{GR04} and the $q$-analogue of Bailey's $_2 F_1 (-1)$ summation formula  \cite[(II.2,10)]{GR04}. Identity \eqref{new3} also can be derived by applying the $q$-binomial theorem   and the $q$-Gauss sum as given in \cite[(II.3), (II.8)]{GR04}. Moreover, the above identities contain many known identities as special cases.  For example, when $z=1$ and  $q\mapsto q^2$, \eqref{new2} reduces to Entry 4.2.11 in \cite{AB09}, which is
\begin{align*}
\sum_{n=0}^{\infty}\frac{(-q;q^2)_n}{(q^4;q^4)_n}q^{n^2+2n}
=\frac{(-q;q^2)_\infty(q,q^5,q^6;q^6)_\infty}{(q^2;q^2)_\infty}.
\end{align*}

\section{Andrews and Warnaar's bilateral Bailey lemmas}

Recall that a series is of Hecke--type if it has the  form
\[
 \sum\limits_{(n,j)\in D} (-1)^{H(n,j)} q^{Q(n,j)+L(n,j)},
 \]
where $H$ and $L$ are linear forms, $Q$ is a quadratic form, and $D$ is some subset of $\mathbb{Z}\times \mathbb{Z}$ such that $Q(n,j)\geq 0$ for any $(n,j)\in D$. Hecke--type series have received extensive attention since the study of Jacobi and Hecke, see \cite{And841, HM14}. In this section, we fit the bilateral Bailey pair \eqref{originalbbp} into the bilateral Bailey lemmas given by Andrews and Warnaar \cite{AW02}, which will imply identities on Hecke--type series.

Andrews and Warnaar \cite{AW02} gave the following forms of  the bilateral Bailey lemmas for specific bilateral Bailey pairs $(\alpha_n, \beta_n)$ related to $a=1$ and $q \mapsto q^2$, which satisfy that
\begin{align*}
\beta_n=\sum_{r=-n}^n\frac{\alpha_r}{(q^2; q^2)_{n-r}(q^2; q^2)_{n+r}}.
\end{align*}

\begin{lemma}[Andrews and Warnaar's bilateral Bailey lemma I]\label{thm5}
If  $(\alpha_n, \beta_n)$ is a bilateral Bailey pair  related to $a=1$ and $q \mapsto q^2$, then
\begin{align*}
\sum_{n=0}^\infty\frac{(q^2; q^2)_{2n}q^n}{(-q)_{2n+1}}\beta_n=\sum_{n=0}^\infty q^{n(n+1)}\sum_{j=-n}^nq^{-j^2}\alpha_j.
\end{align*}
\end{lemma}

\begin{lemma}[Andrews and Warnaar's bilateral Bailey lemma II]\label{thm10}
If  $(\alpha_n, \beta_n)$ is a bilateral Bailey pair  related to $a=1$ and $q \mapsto q^2$, then
\begin{align*}
\sum_{n=0}^\infty(q)_{2n}q^n\beta_n=\sum_{n=0}^\infty q^{\binom{n+1}{2}}\sum_{j=-\lfloor\frac{n}{2}\rfloor}^{\lfloor\frac{n}{2}\rfloor}q^{-2j^2}\alpha_j.
\end{align*}
\end{lemma}

Letting $q\mapsto q^2$ in the bilateral Bailey pair \eqref{originalbbp}, and then inserting it into Lemma \ref{thm5} and Lemma \ref{thm10}, respectively, we can get the following results on Hecke--type series.

\begin{theorem}\label{hecktype}
We have
\begin{align}
&\sum_{n=0}^{\infty}\frac{(z, q^2/z; q^2)_n}{(-q)_{2n+1}}q^{n}=\sum_{n=0}^{\infty}q^{n(n+1)}
\sum_{j=-n}^{n}(-1)^jz^jq^{-j},\label{hecktype1}\\
&\sum_{n=0}^{\infty}\frac{(z, q^2/z; q^2)_n}{(-q)_{2n}}q^{n}=\sum_{n=0}^{\infty}q^{\binom{n+1}{2}}\sum_{j=-\lfloor
\frac{n}{2}\rfloor}^{\lfloor\frac{n}{2}\rfloor}(-1)^jz^jq^{-j^2-j}.\label{hecktype2}
\end{align}
\end{theorem}

In \eqref{hecktype1}, if we take $z\mapsto-z^{-1}q$, it turns to be identity $(1.4)$ in \cite{AW02}. Moreover, if we let $z\mapsto q$, it leads to \cite[(1.1b)]{AW02}; if we take $z\mapsto -q$, it reduces to the last identity in \cite[P. 181]{AW02}, which is
\begin{align*}
\sum_{n=0}^\infty\frac{(-q;q^2)_n^2}{(-q)_{2n+1}}q^n=\sum_{n=0}^\infty(2n+1)q^{n(n+1)}.
\end{align*}
In \eqref{hecktype2}, by letting $z\mapsto\pm q$, we get
\begin{align*}
\sum_{n=0}^\infty\frac{(\pm q; q^2)_n^2}{(-q)_{2n}}q^n=\sum_{n=0}^\infty q^{\binom{n+1}{2}}\sum_{j=-\lfloor\frac{n}{2}\rfloor}^
{\lfloor\frac{n}{2}\rfloor}(\mp1)^jq^{-j^2}.
\end{align*}

Taking $q\mapsto q^2$ in the bilateral Bailey pair \eqref{newbbaileypair1}, and then inserting it into Lemma \ref{thm5} and Lemma \ref{thm10}, respectively, we obtain the following results.

\begin{theorem}\label{hecktypeite1}
We have
\begin{align}
&\sum_{n=0}^{\infty}\sum_{j=0}^n\frac{(q)_{2n}(z, q^4/z; q^4)_jq^{n+2j}}{(1+q^{2n+1})(1+q^{4j})(q^4;q^4)_{n-j}(q^2;q^2)_{2j}}
=\sum_{n=0}^{\infty}q^{n(n+1)}\sum_{j=-n}^{n}\frac{(-1)^jz^jq^{j^2}}{1+q^{4j}},\label{hecktypeite11}\\
&\sum_{n=0}^{\infty}\sum_{j=0}^n\frac{(q)_{2n}(z, q^4/z; q^4)_jq^{n+2j}}{(1+q^{4j})(q^4;q^4)_{n-j}(q^2;q^2)_{2j}}=\sum_{n=0}^{\infty}q^{\binom{n+1}{2}}\sum_{j=-\lfloor
\frac{n}{2}\rfloor}^{\lfloor\frac{n}{2}\rfloor}\frac{(-1)^jz^j}{1+q^{4j}}.\label{hecktypeite12}
\end{align}
\end{theorem}
Note that when $z=1$ or $z=q^4$ in the above results, it gives the following identities on Hecke--type series
\begin{align*}
&\sum_{n=0}^{\infty}\frac{(q;q^2)_nq^n}{(1+q^{2n+1})(-q^2;q^2)_n}
=\sum_{n=0}^{\infty}q^{n(n+1)}\Big(1+2\sum_{j=1}^{n}(-1)^jq^{j^2}\Big),\\
&\sum_{n=0}^{\infty}\frac{(q;q^2)_nq^n}{(-q^2;q^2)_n}
=\sum_{n=0}^{\infty}q^{\binom{n+1}{2}}\Big(1+2\sum_{j=1}^{\lfloor\frac{n}{2}\rfloor}(-1)^j\Big).
\end{align*}
By using the bilateral Bailey lattice as given in Lemma \ref{wnewbbf2}, we get the following results.
\begin{theorem}\label{hecktypeite2}
We have
\begin{align}
&\sum_{n=0}^{\infty}\sum_{j=0}^n\frac{(q)_{2n}(z, q^4/z; q^4)_j}{(1+q^{2n+1})(q^4;q^4)_{n-j}(q^2;q^2)_{2j}}q^{3n-2j}
=\sum_{n=0}^{\infty}q^{n(n+1)}\sum_{j=-n}^{n}(-1)^jz^jq^{j^2-2j},\label{hecktypeite21}\\
&\sum_{n=0}^{\infty}\sum_{j=0}^n\frac{(q)_{2n}(z, q^4/z; q^4)_j}{(q^4;q^4)_{n-j}(q^2;q^2)_{2j}}q^{3n-2j}=\sum_{n=0}^{\infty}q^{\binom{n+1}{2}}\sum_{j=-\lfloor
\frac{n}{2}\rfloor}^{\lfloor\frac{n}{2}\rfloor}(-1)^jz^jq^{-2j}.\label{hecktypeite22}
\end{align}
\end{theorem}

\begin{proof}
Taking $q\mapsto q^2$ in the bilateral Bailey pair \eqref{originalbbp}, and then inserting it into Lemma \ref{wnewbbf2}, we obtain that
\begin{align}
(\alpha_n,\ \beta_n)=\Big((-1)^nz^nq^{n(n-1)},\ \sum_{j=0}^n\frac{(z,q^2/z;q^2)_j}{(q^2;q^2)_{n-j}(q)_{2j}}q^{n-j}\Big)\label{newbbaileypair2}
\end{align}
forms a bilateral Bailey pair related to $a=1$. Substituting $q\mapsto q^2$ and inserting it into Lemma \ref{thm5} and Lemma \ref{thm10}, respectively, we   obtain the desired identities.
\end{proof}

When  $z=1$ or $z\mapsto q^4$ in \eqref{hecktypeite21}, it leads to
\begin{align*}
\sum_{n=0}^{\infty}\frac{(q;q^2)_n}{(1+q^{2n+1})(-q^2;q^2)_n}q^{3n}
=\sum_{n=0}^{\infty}q^{n(n+1)}\sum_{j=-n}^{n}(-1)^jq^{j^2-2j}.
\end{align*}
When $z=1$ or $z\mapsto q^4$ in \eqref{hecktypeite22}, we obtain that
\[
\sum_{n=0}^{\infty}\frac{(q;q^2)_{n}}{(-q^2;q^2)_{n}}q^{3n}=\sum_{n=0}^{\infty}q^{\binom{n+1}{2}}\sum_{j=-\lfloor
\frac{n}{2}\rfloor}^{\lfloor\frac{n}{2}\rfloor}(-1)^jq^{-2j}.
\]
If we take $z\mapsto -q^2$ in \eqref{hecktypeite22} and consider the parities of $n$, then it gives that
\begin{align*}
&\sum_{n=0}^{\infty}\sum_{j=0}^n\frac{(q)_{2n}(-q^2; q^4)_j^2}{(q^4;q^4)_{n-j}(q^2;q^2)_{2j}}q^{3n-2j}
=\sum_{n=0}^\infty (2n+1) (1+q^{2n+1})q^{2n^2+n}.
\end{align*}

\noindent {\small {\bf Acknowledgments.}   This work is supported by the National Natural
Science Foundation of China (No. 12071235) and the Fundamental Research Funds for the Central
%Universities of China.

\bibliographystyle{amsplain}

%%===========================================================================================%%
%% If you are submitting to one of the Nature Portfolio journals, using the eJP submission   %%
%% system, please include the references within the manuscript file itself. You may do this  %%
%% by copying the reference list from your .bbl file, paste it into the main manuscript .tex %%
%% file, and delete the associated \verb+\bibliography+ commands.                            %%
%%===========================================================================================%%

%\bibliography{sn-bibliography}% common bib file
%% if required, the content of .bbl file can be included here once bbl is generated
%%\input sn-article.bbl

\end{document}